\documentclass[a4paper]{article}
 \usepackage{amsthm,amssymb,amsmath,enumerate,graphicx,psfrag, epsfig}

 \newtheorem{definition}{Definition}

 \newtheorem{thm}[definition]{Theorem}
 \newtheorem{corollary}[definition]{Corollary}
 
\newtheorem{lemma}[definition]{Lemma} 
 \newtheorem{conjecture}[definition]{Conjecture}
 
 \theoremstyle{remark}

 \theoremstyle{remark}

 \newcommand{\emtext}[1]{\text{\em #1}}
\newcommand{\minae}{\textrm{\ae}}
\newcommand{\minAE}{\textrm{\AE}}

 \title{The Loebl--Koml\'os--S\'os conjecture for trees of diameter $5$ and for certain caterpillars
 }
\author{Diana Piguet\thanks{Supported by ITI grant no.~1M00216220808.} \and Maya Jakobine Stein\thanks{Supported by  FAPESP grant no.~05/54051-9.}}
\date{}
 
\begin{document}
 \maketitle
\begin{abstract}
Loebl, Koml\'os, and S\'os conjectured that if at least half the
vertices of a graph $G$ have degree at least some $k\in \mathbb N$, then
every tree with at most~$k$ edges is a subgraph of $G$.

We prove the conjecture for all trees of diameter at most~$5$ and for a class of caterpillars.
Our result implies a bound on the Ramsey number $r(T,T')$ of trees~$T,T'$ from the above classes.
\end{abstract}

\section{Introduction}
Loebl conjectured (see~\cite{discrepency}) that if $G$ is a graph of order $n$, and at least $n/2$ vertices of $G$ have degree at least $n/2$, then every tree with at most $n/2$ edges is a subgraph of $G$.
Koml\'os and S\'os generalised his conjecture to the following.

\begin{conjecture}[Loebl--Koml\'os--S\'os conjecture~\cite{discrepency}]\label{conj:LKS}
Let $k,n\in\mathbb N$, and let  $G$ be a graph of order $n$ so that  at least $n/2$ vertices of $G$ have degree at least  $k$. Then every tree with at most $k$ edges is a subgraph of $G$.
\end{conjecture}

The conjecture is asymptotically correct, at least if $k\in\Theta(n)$:
In~\cite{DM}, the authors of this paper prove an approximate version of the Loebl--Koml\'os--S\'os conjecture for large $n$, and $k$ linear in $n$. In Loebl's original form, an approximate version has been shown by Ajtai, Koml\'os and Szemer\'edi~\cite{aks}, and recently, Zhao~\cite{zhao} has shown the exact version.

The bounds from the conjecture could not be significantly lower. It is easy to see that we need at least one vertex of degree at least $k$ in $G$.
On the other hand, the amount of vertices of large degree that is required in Conjecture~1 is necessary.
We shall discuss the bounds in more detail in Section~\ref{sec-bound}.

Conjecture~1 trivially holds
for stars. In order to see  the conjecture for trees that consist of two stars with adjacent centres, it is enough to realise that $G$ must have two adjacent vertices of degree at least $k$. Indeed, otherwise  one easily reaches a contradiction by double-counting the number of edges between the set $L\subseteq V(G)$ of vertices of degree at least $k$, and the set $S:=V(G)\setminus L$.

Hence, the Loebl--Koml\'os--S\'os conjecture is true for all trees of diameter at most~$3$.
Barr and Johansson~\cite{Barr}, and independently Sun~\cite{Sun}, proved the conjecture for all trees of diameter  $4$.
Our main result is a proof of Conjecture~1 for all trees of diameter at most~$5$.

\begin{thm}\label{thm:LKS-5}
Let $k,n\in\mathbb N$, and let  $G$ be a graph of order $n$ so that  at least $n/2$ vertices of $G$ have degree at least  $k$. Then every tree of diameter at most~$5$ and with at most~$k$ edges is a subgraph of $G$.
\end{thm}

Paths and path-like trees constitute another class of trees for which Conjecture~1 has been studied.
Bazgan, Li, and Wo\'zniak~\cite{blw}
proved the conjecture for paths and for all trees that can be obtained from a path and a  star by identifying one of the vertices of the path with the centre of the star.

We extend their result to a larger class of trees, allowing for two stars instead of one, under certain restrictions. Let $\mathcal T(k,\ell,c)$ be the class of all trees with $k$ edges which can be obtained from a path $P$ of length $k-\ell$, and two stars $S_1$ and~$S_2$ by identifying the centres of the $S_i$ with two vertices that lie at distance~$c$ from each other on $P$.

\begin{thm}\label{prop:2stars}
Let $k,\ell,c,n\in\mathbb N$ such that $\ell\geq c$. Let $T\in\mathcal T(k,\ell,c)$, and let $G$ be a graph of order $n$ so that at least $n/2$ vertices of $G$ have degree at least $k$. If~$c$ is even, or $\ell+c\geq \lfloor n/2\rfloor $
 (or both), then
$T$ is a subgraph of~$G$.
\end{thm}

If true, Conjecture~1 has an interesting application in Ramsey theory, as has been first observed in~\cite{discrepency}.
The Ramsey number $r(T_{k+1},T_{m+1})$ of two trees~$T_{k+1},T_{m+1}$  with $k$, resp.~$m$ edges is defined as the minimal integer $n$ so that any colouring of the edges of the complete graph $K^n$ of order $n$ with two colours, say red and blue, yields either a red copy of $T_{k+1}$, or a blue copy of~$T_{m+1}$ (or both).

Observe that in any such colouring, either the red subgraph of $K^n$ has at least~$n/2$ vertices of degree at least~$k$, or the blue subgraph has at least $n/2$ vertices of degree at least $n-k$. Hence,
 if Conjecture~1 holds for all $k$ and $n$, then $r(T_{k+1},T_{m+1})\leq k+m$ for all $k,m \in\mathbb N$.

The bound $k+m$  is asymptotically true: the authors of this article prove in~\cite{DM} that $r(T_{k+1},T_{m+1})\leq k+m+o(k+m)$, provided that~$k,m\in \Theta(n)$. Although particular classes of trees (such as paths~\cite{gegy}) have smaller Ramsey numbers, the bound $k+m$ would be tight in the class of all trees. In fact, the Ramsey number of two stars
  with $k$, resp.~$m$ edges, is~$k+m-1$, if both $k$ and~$m$ are even, and $k+m$ otherwise~\cite{harary}.

Our results on Conjecture~1 allow us to bound the 
Ramsey numbers of further classes of trees. Theorem~\ref{thm:LKS-5} and Theorem~\ref{prop:2stars} have the following corollary.

\begin{corollary}\label{cor:LKS-5}
Let $T_1, T_2$ be trees with $k$ resp.~$m$ edges such that, for $i=1,2$, either $T_i$ is as in Theorem~\ref{prop:2stars} or has diameter at most~$5$ (or both).
Then  $r(T_1,T_2)\leq k+m$.
\end{corollary}

\section{Notation}

Throughout the paper, $\mathbb{N}=\mathbb{N}_+$.

Our graph-theoretic notation follows~\cite{diestelBook05}, let us here review the main definitions needed.
A graph~$G$ has vertex set $V(G)$ and edge set $E(G)$.
As we will not distinguish between isomorphic graphs we consider a graph $H$ to be a  subgraph of $G$,  if there exists an injective mapping from $V(H)$ to $V(G)$ which preserves adjacencies. We shall then write $H\subseteq G$, and call any mapping as above an {\em embedding} of $V(H)$ in $V(G)$.

The {\em neighbourhood} of a vertex $v$ is $N(v)$, and the neighbourhood of a set $X\subseteq V(G)$ is $N(X):=\bigcup_{v\in X}N(v) \setminus X$. We set $\deg_X(v):=|N(v)\cap X|$ and 
$\deg(v):=\deg_{V(G)}(v)$.

The {\em length} of a path is the number of its edges. For a path $P$ and two vertices $x,y\in V(P)$, let $xPy$ denote the subpath of $P$ which starts in $x$ and ends in~$y$. We define $ xP$ and $Py$ analogously.
The {\em distance} between two vertices is the length of the shortest path connecting them. The {\em diameter} of~$G$ is the longest distance between any two vertices of $G$.

\section{Discussion of the bounds}\label{sec-bound}

Let us now discuss the
 bounds in  Conjecture~1. On one hand, as $T$ could be a star, it is clear that we need that $G$ has a vertex of degree at least $k$.

On the other hand, we also need a certain amount of vertices of large degree. In fact,
 the amount $n/2$ we require cannot be lowered by a factor of $(k-1)/(k+1)$.
We shall show now that if we require only  $\frac{k-1}{k+1} n/2 = n/2-n/(k+1)$ vertices to have degree at least~$k$, the conjecture becomes false whenever $k+1$ is even and divides $n$.

To see this, construct a graph $G$ on $n$ vertices as follows. Divide $V(G)$ into $2n/(k+1)$ sets $A_i$, $B_i$, so that $|A_i|=(k-1)/2$, and $|B_i|=(k+3)/2$, for $i=1,\ldots ,n/(k+1)$. Insert all edges inside each $A_i$, and insert all edges between each pair $A_i$, $B_i$. Now, consider the tree $T$ we obtain from a star with $(k+1)/2$ edges by subdividing each edge but one. Clearly, $T$ is not a subgraph of $G$.

A similar construction shows that we need more than $\frac n2-\frac {2n}{k+1}$ vertices of large degree, when $k+1$ is odd and divides $n$, and furthermore,  by adding some isolated vertices, our example can  be  modified for arbitrary $k$. This shows that at least $n/2-2\lfloor n/(k+1)\rfloor - (n\mod (k+1))$ vertices of large degree are needed. Hence, when $\max\{n/k, n\mod k\}\in o(n)$,
 the bound $n/2$ is asymptotically best possible.

\section{Trees of small diameter}

In this section, we prove Theorem~\ref{thm:LKS-5}. We shall prove the theorem by contradiction.
So, assume that there are  $k,n\in\mathbb N$, and
a graph $G$ with $|V(G)|=n$, such that at least $n/2$ vertices of $G$ have degree at least $k$. Furthermore, suppose that $T$ is a tree  of diameter at most $5$ with $|E(T)|\leq k$ such that $T\not\subseteq G$.

We may assume that among all such counterexamples~$G$ for $T$, we have chosen~$G$ edge-minimal. In other words, we assume that the deletion of any edge of $G$ results in a graph which has less than $n/2$ vertices of degree $k$.

Denote by $L$ the set of those vertices of $G$ that have degree at least $k$, and set $S:=V(G)\setminus L$. Observe that, by our edge-minimal choice of $G$, we know that~$S$ is independent. Also, we may assume that $S$ is not empty.

Clearly, our assumption that  $T\not\subseteq G$ implies that for each set $M$ of leaves of $T$ it  holds that
\begin{equation}\label{eq:leaves}
 \text{there is no embedding }\varphi\text{ of }V(T)\setminus M\text{ in }V(G) \text{ so that }\varphi (N(M))\subseteq L.
\end{equation}

In what follows, we shall often use the fact that both the degree of a vertex and the cardinality of a set of vertices are natural numbers. In particular, assume that $U,X\subseteq V(G)$, $u\in V(G)$, and  $x\in \mathbb{Q}$. Then the following implication holds.
\begin{equation}\label{eq:x}
\emtext{ If }|U|<x+1\emtext{  and }\deg_X(u)\geq x,\emtext{ then }|U|\leq \deg_X(u).
\end{equation}

Let us now define a useful partition of $V(G)$. Set
\begin{align*}
A&:=\{v\in L: \deg_{L}(v)< \frac k2\},\\
B&:=L\setminus A,\\
C&:=\{v\in S: \deg(v)=\deg_{L}(v)\geq \frac k2\},\emtext{ and }\\
D&:= S\setminus C.
\end{align*}

Let $r_1r_2\in E(T)$ be such that each vertex of $T$ has distance at most $2$ to at least one of $r_1$, $r_2$.
 Set
\begin{align*}
V_1&:=N(r_1)\setminus \{r_2\},
& V_2:=N(r_2)\setminus \{r_1\},\\
W_1&:=N(V_1)\setminus \{r_1\},
& W_2:=N(V_2)\setminus \{r_2\}.
\end{align*}
Furthermore, set
\begin{equation*}
V_1':=N(W_1)\ \ \text{ and }\ \
V_2':=N(W_2).
\end{equation*}

Observe that $|V_1\cup V_2\cup W_1\cup W_2|< k$.
So, without loss of generality (since we can otherwise interchange the roles of $r_1$ and $r_2$), we may assume that
\begin{equation}\label{eq:RQ}
|V_2\cup W_1|<\frac k2.
\end{equation}
Since $|V_1'|\leq |W_1|$, this implies that
\begin{equation}\label{eq:P'Q}
|V_1'\cup V_2|<\frac k2.
\end{equation}

Now, assume that there is an edge $uv\in E(G)$ with $u,v\in B$. We shall conduct this assumption to a contradiction to~\eqref{eq:leaves} by proving that then we can define an embedding $\varphi$ so that $\varphi(V_1'\cup V_2\cup\{r_1,r_2\})\subseteq L$.

Define the embedding $\varphi$ as follows.
Set $\varphi(r_1):=u$, and set $\varphi(r_2):=v$. Map~$V_1'$ to a subset of~$N(u)\cap L$, and~$V_2$ to a subset of~$N(v)\cap L$. This is possible, as~\eqref{eq:x} and~\eqref{eq:P'Q} imply that $|V_1'\cup V_2|+1\leq \deg_L(v)$.

We have thus reached the desired contradiction to~\eqref{eq:leaves}. This proves that
\begin{equation}\label{eq:B}
 B\emtext{ is independent.}
\end{equation}

Set $$N:=N(B)\cap L\subseteq A.$$

We claim that each vertex $v\in N$ has degree

\begin{equation}\label{eq:deg_B>k/4}
\deg_{B}(v)< \frac k4.
\end{equation}

Then,~\eqref{eq:B} and~\eqref{eq:deg_B>k/4} together imply that

$$|B|\frac k2\leq e(N,B)\leq |N|\frac k4,$$

and hence,

\begin{equation}\label{eq:N}
|N|\geq 2|B|.
\end{equation}

In order to see~\eqref{eq:deg_B>k/4}, suppose otherwise, i.\,e., suppose that there is a vertex $v\in N$ with $\deg_{B}(v)\geq \frac k4$.
Observe that by~\eqref{eq:P'Q},  $|V_1'\cup V_2'|<\frac k2$ and hence we may assume that at least one of $|V_1'|$, $|V_2'|$, say $|V_1|$, is smaller than $\frac k4$. The case when $|V_2'|<\frac k4$ is done analogously.

We define an embedding $\varphi$ of $V_1'\cup V_2'\cup\{r_1,r_2\}$ in~$V(G)$ as follows. Set $\varphi(r_1):=v$ and map $V_1'\cup \{r_2\}$ to $N(v)\cap B$. This is possible, because
$|V_1'|+1<\frac k4+1$, and thus, by~\eqref{eq:x}, $|V_1'\cup \{r_2\}|\leq \deg_B(v)$.

Next, map $V_2'$ to $N(u)\cap L$, where $u:=\varphi (r_2)$. This is safe, as~\eqref{eq:B} implies that $$N(u)\cap L\cap \varphi(V_1')=\emptyset,$$ and furthermore, by~\eqref{eq:P'Q},  $|V_2'|+1< \frac k2+1$. Together with~\eqref{eq:x}, we thus obtain that $$|V_2'\cup \{r_1\}|=|V_2'|+1\leq \deg_L(u).$$

 This yields  the desired contradiction to~\eqref{eq:leaves}, and thus  proves~\eqref{eq:deg_B>k/4}.

\medskip

Now, set
\begin{equation*}
X:=\{v\in L: \deg_{C\cup L}(v)\geq \frac k2\}\supseteq B.
\end{equation*}

We claim that

\begin{equation}\label{eq:X-C_edge}
e(X,C)=0.
\end{equation}

Observe that then
\begin{equation}\label{eq:X=B}
X=B,
\end{equation}

and,
\begin{equation}\label{eq:BC}
e(B,C)=0.
\end{equation}

In order to see~\eqref{eq:X-C_edge}, suppose for contradiction that
there exists an edge $uv$ of~$G$ with $u\in X$ and $v\in C$. We define an embedding $\varphi $ of $V_1'\cup V_2\cup W_1^C\cup \{r_1,r_2\}$ in~$V(G)$, where $W_1^C$ is a certain subset of $W_1$, as follows.

Set $\varphi(r_1):=u$, and set $\varphi(r_2):=v$. Embed a subset $V_1^C$ of $V_1'$ in $N(u)\cap C$, and a subset $V_1^L=V_1'\setminus V_1^C$ in $N(u)\cap L$. We can do so by~\eqref{eq:x}, and since by~\eqref{eq:P'Q}, $|V_1'|<\frac k2$.

Next, map $W_1^C:=N(V_1^C)\cap W_1$ and $V_2$ to $L$,
 preserving all adjacencies. This is possible by~\eqref{eq:x}.
Indeed, observe that by the independence of $S$, each vertex in~$C$ has at least $\frac k2$ neighbours in $L$, while  by~\eqref{eq:RQ}, we have that

$$|V_1^L\cup W_1^C\cup V_2\cup \{u\}|\leq |W_1\cup V_2|+1<\frac k2+1.$$

We have hence mapped $V_1',V_2,W_1^C$ and the vertices $r_1$ and $r_2$ in a way so that the neighbours of $(V_1\setminus V'_1)\cup (W_1\setminus W_1^C)\cup W_2$ are mapped to $L$. This yields the desired contradiction to~\eqref{eq:leaves}. We have thus shown~\eqref{eq:X-C_edge}, and consequently, also~\eqref{eq:X=B} and~\eqref{eq:BC}.

\medskip

Observe that
\begin{equation}\label{eq:D}
D\neq \emptyset.
\end{equation}
 Indeed, otherwise $C\neq \emptyset$ and thus by~\eqref{eq:X-C_edge}, we have that $A\neq \emptyset$. By~\eqref{eq:X=B}, this implies  that $D\neq\emptyset$, contradicting our assumption.

Next, we claim that there is a vertex $w\in N$ with
\begin{equation}\label{eq:config-1}
\deg_{C\cup L}(w)\geq \frac k4.
\end{equation}

Indeed, suppose otherwise.
Using~\eqref{eq:X=B} and~\eqref{eq:D}, we obtain that

$$|A\setminus N|\frac k2+|N|\frac {3k}4\leq e(A,D)< |D|\frac k2.$$

Dividing by $\frac k4$, it follows that
\begin{equation*}\label{eq:2D}
2|A|+|N|< 2|D|.
\end{equation*}

Together with~\eqref{eq:N}, this yields

$$|D|>|A|+|B|\geq \frac n2,$$

a contradiction, since by assumption $|D|\leq |S|\leq \frac n2$.
This proves~\eqref{eq:config-1}.

\medskip

Using a similar argument as for~\eqref{eq:X-C_edge}, we can now show that
\begin{equation}\label{eq:B-C}
|V_1'|\geq \frac k4.
\end{equation}

Indeed, otherwise use~\eqref{eq:config-1} in order to map $r_1$ to $w$, $r_2$ to any $u\in N(w)\cap B$, and embed~$V_1'$ in $C\cup L$,
and~$V_2$ and~$W_1^C$ (defined as above) in~$L$, preserving adjacencies. This yields the desired contradiction to~\eqref{eq:leaves}.

\medskip

Observe that~\eqref{eq:B-C} implies that $\frac k4\leq |V_1'|\leq |W_1|$, and hence, by~\eqref{eq:RQ}, \begin{equation}\label{eq:V2}
|V_2|<\frac k4.
\end{equation}

We claim that moreover

\begin{equation}\label{eq:P'US}
|V_1'\cup W_2|\geq\frac k2.
\end{equation}

Suppose for contradiction that this is not the case. We shall then define an embedding $\varphi$ of $V_1'\cup V_2\cup \{r_1,r_2\}\cup W_2^C$ in $V(G)$, for a certain $W_2^C\subseteq W_2$, as follows.

Set~$\varphi(r_2):=w$, and choose for $\varphi(r_1)$ any vertex $u\in N(w)\cap B$. Map a subset~$V_2^C$ of $V_2$ to $N(w)\cap C$, and map $V_2^L:=V_2\setminus V_2^C$ to $N(w)\cap L$. This is possible, as by~\eqref{eq:x}, by~\eqref{eq:config-1}, and by~\eqref{eq:V2}, we have that $\deg_{C\cup L}(w)\geq |V_2|+1$.

We may assume that we chose $V_2^C$ so that it contains as many vertices from $V_2\setminus V_2'$ as possible. We now plan to map $W_2^C:=N(V_2^C)\cap W_2$ to $V(G)$. Observe that by our choice of $V_2^C$, either $V^C_2\subseteq V_2\setminus V_2'$, and hence $W^C_2$ is empty, or $V_2^L\subseteq V_2'$. In the latter case, it follows that
\begin{equation*}\label{eq:Q_C}
|V_2^L|\leq |W_2\setminus W_2^C|,
\end{equation*}

and by our assumption that $|V_1'\cup W_2|<\frac k2$, we obtain that $$|V_2^L\cup W_2^C\cup \{r_2\}|\leq |W_2|+1<\frac k2+1.$$

 Thus,  by~\eqref{eq:x}, for each $v\in C$, we have that $\deg(v)\geq |V_2^L\cup W_2^C|+1$.
Observe that~\eqref{eq:BC} implies that $u\notin N(C)$. So, we can  map~$W_2^C$ to $L$, preserving all adjacencies.

Next, we shall map $V_1'$ to $N(u)\cap L$. We can do so, since $$|V'_1\cup V_2^L\cup \{ r_2\}|<\frac k2 +1$$ by~\eqref{eq:P'Q}, and since, in the case that $W^C_2=\emptyset$, we use  by our assumption that~\eqref{eq:P'US} does not hold
to see that $$|V_1'\cup V_2^L\cup W_2^C\cup \{r_2\}|\leq |V_1'\cup W_2|+1<\frac k2+1.$$
Hence, by~\eqref{eq:x}, in either case it follows that $$\deg_L(u)\geq |V_1'\cup V_2^L\cup W_2^C\cup \{r_2\}|.$$

We have thus embedded all of $V(T)$ except $(V_1\setminus V_1')\cup (W_2\setminus W_2^C)\cup W_1$ whose neighbours have their image in $L$.
This yields a contradiction to~\eqref{eq:leaves}, and hence proves~\eqref{eq:P'US}.

\medskip

Now, let $x\in \mathbb Q$ be such that $|V_1'|=x\cdot \frac k2$. Then, by~\eqref{eq:P'US},
\begin{equation*}\label{eq:S-big}
|W_2|\geq (1-x)\frac k2.
\end{equation*}

On the other hand,
since $|W_1|\geq |V_1'|=x\frac k2$, and $|V(T)|=k+1$, we have that
\begin{align*}
|(V_1\setminus V_1') \cup V_2\cup W_2|& =|V(T)\setminus (V_1'\cup W_1\cup \{r_1,r_2\})|\\ &<(1-x)k.
\end{align*}

Combining these inequalities, we obtain that
\begin{align}\label{eq:QP}
|V_1\cup V_2| &=|V_1'\cup (V_1\setminus V_1')\cup V_2|\notag \\
&<x\frac k2+(1-x)k-(1-x)\frac k2 \notag \\
& =\frac k2.
\end{align}

The now gained information on the structure of  $T$  enables us to show next that for each vertex $v$ in $\tilde N:=N(B\cup C)\cap L$ it holds that
\begin{equation}\label{eq:no-vertex}
\deg_L(v)< \frac k4.
\end{equation}

Suppose for contradiction that this is not the case, i.\,e., that there exists a~$v\in \tilde N$ with $\deg_L(v)\geq \frac k4$. We define an embedding $\varphi$ of~$V(T)\setminus (W_1\cup W_2)$ in~$V(G)$ so that $N(W_1\cup W_2)$ is mapped to $L$.

Set $\varphi(r_2):=v$ and choose for $\varphi(r_1)$ any vertex $u\in N(v)\cap (B\cup C)$. By~\eqref{eq:V2}, and since we assume that~\eqref{eq:no-vertex} does not hold, we can map $V_2$ to $N(v)\cap L$. Moreover, since by~\eqref{eq:x} and~\eqref{eq:QP} we have that
$$\deg_L(u)\geq |V_1\cup V_2\cup \{r_2\}|,$$
 we can map~$V_1$ to $N(u)\cap L$.
We have hence mapped all of $V(T)$ but $W_1\cup W_2$  to $L$, which yields the desired contradiction to~\eqref{eq:leaves} and thus establishes~\eqref{eq:no-vertex}.

\medskip

We shall finally bring~\eqref{eq:no-vertex} to a contradiction. We use~\eqref{eq:B},~\eqref{eq:X=B},~\eqref{eq:BC} and~\eqref{eq:no-vertex} to obtain that
\begin{align*}
|D|\frac k2 & \geq e(D,L)\\
 & \geq |A\setminus \tilde N|\frac k2+|\tilde N|\frac {3k}{4}-e(C,\tilde N)+|B|k-e(B,\tilde N)\\
 & \geq |A|\frac k2+|\tilde N|\frac k4 + |B|k - e(B\cup C,\tilde N).
\end{align*}

Since $|S|\leq |L|$ by assumption, this inequality implies that
\begin{align*}
|B|\frac k2 +|C|\frac k2+|\tilde N|\frac k4 &\leq |B|\frac k2+(|A|+|B|-|D|)\frac k2+|\tilde N|\frac k4\\  &\leq e(B\cup C,\tilde N)\\ & \leq |\tilde N|\frac k2,
\end{align*}

where the last inequality follows from the fact that $\tilde N\subseteq A=L\setminus X$, by~\eqref{eq:X=B}.

Using~\eqref{eq:no-vertex},
a final double edge-counting now gives

\begin{align*}
(|A|+|B|+|C|)\frac k2 & \leq |A|\frac k2+|\tilde N|\frac k4\\ & \leq e(A,S)\\ & <|D|\frac k2+|C| k\\ & =|S|\frac k2+|C|\frac k2,
\end{align*}

implying that
$|L|<|S|$, a contradiction. This completes the proof of Theorem~\ref{thm:LKS-5}.

\section{Caterpillars}

In this section, we shall prove Theorem~\ref{prop:2stars}. We shall actually prove something stronger, namely 
  Lemmas~\ref{cor:path-2} and~\ref{cor:path-3}.

 A {\em caterpillar} is a tree $T$ where each vertex has distance at most~$1$ to some central path $P\subseteq T$.
In this paper, we shall consider a special subclass of caterpillars, namely those that have at most two vertices of degree greater than~$2$. Observe that any such caterpillar $T$ can be obtained from a path~$P$ by identifying two of its vertices, $v_1$ and $v_2$, with the centres of stars.
We shall write~$T=C(a,b,c,d,e)$, if $P$ has length $a+c+e$, and $v_1$ and $v_2$ are the $(a+1)$th and $(a+c+1)$th vertex on $P$, and have  $b$, resp.~$d$, neighbours outside~$P$. Therefore, if $a,e>0$, then $C(a,b,c,d,e)$ has $b+d+2$ leaves.

We call $P$ the {\em body}, and~$v_1$ and $v_2$ the {\em joints} of the caterpillar. For illustration, see Figure~\ref{im-path1}.

\begin{figure}[h!]\begin{center}\epsfxsize .5\hsize\epsfbox{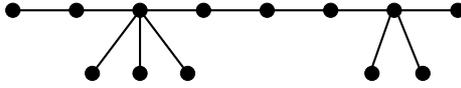}\caption{The caterpillar $C(2,3,4,2,1)$ or $ C(2,3,4,3,0)$.
}\label{im-path1}\end{center}\end{figure}

So $\mathcal T(k,\ell,c)$, as defined in the introduction, denotes the class of all caterpillars  $C(a,b,c,d,e)$ with $b+d=\ell$, and $a+b+c+d+e=k$.
We can thus state the result of Bazgan,~Li, and Wo\'zniak mentioned in the introduction as follows.

\begin{thm}[Bazgan,~Li,~Wo\'zniak~\cite{blw}]\label{cor:path-1}
Let $k,\ell,c\in\mathbb N$, and let $T=C(a,0,c,d,e)$ be a tree from $\mathcal T(k,\ell,c)$. Let $G$ be a graph so that at least half of the vertices of~$G$ have degree at least $k$. Then $T$ is a subgraph of $ G$.
\end{thm}

We shall now prove (separately) the two subcases of  Theorem~\ref{prop:2stars}.
Our first lemma deals with the case when $c$ is even.

\begin{lemma}\label{cor:path-2}
Let $k,\ell,c\in\mathbb N$ so that $c$ is even and $\ell\geq c$. Let  $T\in\mathcal T(k,\ell,c)$, and let $G$ be a graph such  that at least half of the vertices of $G$ have degree at least~$k$.
Then $T$ is a subgraph of $G$.
\end{lemma}

\begin{proof}
Observe that we may assume that $\ell \geq 2$.  Let $v_1$ and $v_2$ be the joints of~$T$, and let $P$ be its body.
As above, denote by $L$ the set of those vertices of~$G$ that have degree at least $k$ and set $S:=V(G)\setminus L$.  We may assume that~$S$ is independent.

By Theorem~\ref{cor:path-1}, there is a path $P_k$ of length $k$ in $G$. Let $\varphi$ be an embedding of $V(P)$ in $V(P_k)$ which maps the starting vertex of $P$ to the starting vertex $u$ of $P_k$. 
 Now, if both $u_1:=\varphi(v_1)$ and $u_2:=\varphi(v_2)$ are in $L$, then we can easily extend~$\varphi$ to~$V(T)$.

On the other hand, if both $u_1$ and $u_2$ lie in $S$, then we can
`shift' the image of~$V(P)$ along~$P_k$ by $1$ away from $u$, i.\,e., we map each vertex~$v\in \varphi(V(P))$ to its neighbour on~$vP_k$. Then, the composition of $\varphi$ and the shift maps both $v_1$ and $v_2$ to $L$,
and can thus be extended to an embedding of $V(T)$.

To conclude, assume that one of the two vertices $u_1$ and $u_2$ lies in $L$ and the
other lies in $S$. As $c$ is  even and $S$ is independent, it
follows that there are two consecutive vertices~$w_1$ and $w_2$ (in this order) on
$u_1P_ku_2$ which lie in $L$.

Similarly as above, shift the image of $V(P)$ away from $u$. In fact, we repeat this shift until  $u_1$ is shifted to $w_1$. If  the composition of $\varphi$ with the iterated shift maps $v_2$ to $L$, we are
done.
Otherwise, we shift the image of $V(P)$
once more. Then both
$v_1$  and $v_2$ are mapped to~$L$, and we are done.

Observe that in total, we have shifted the image of~$V(P)$ at most $c$ times. We could do so, since $|P_k|-|P|=\ell\geq c$ by assumption.
\end{proof}

We now allow $c$ to be odd, restricting the choice of~$k$.
\begin{lemma}\label{cor:path-3}
Let $k,\ell,c, n\in\mathbb N$ be such that $\ell\geq c$.
 Let $T=C(a,b,c,d,e)$ be a tree in  $\mathcal T(k,\ell,c)$, and let~$G$ be a graph of order $n$ such that at least $n/2$  vertices of~$G$ have
degree at least $k$. Suppose that
\begin{enumerate}[(i)]
\item $k\geq \lfloor n/2\rfloor +2\min\{a,e\}$, if $\max\{a,e\}\leq k/2$, and
\item $k\geq \lfloor n/4\rfloor +a+e+1$, if $\max\{a,e\}>k/2$.
\end{enumerate}
Then $T$ is a subgraph of $ G$.
\end{lemma}

Observe that in case~(ii) of Lemma~\ref{cor:path-3} it follows that
$$k> \lfloor n/4\rfloor + \min\{a,e\} + \max\{a,e\}> \lfloor n/4\rfloor + \min\{a,e\} + k/2,$$
and hence, as in (i),
$$k\geq \lfloor n/2\rfloor +2\min\{a,e\}.$$

\begin{proof}[Proof of Lemma~\ref{cor:path-3}] As before, set $L:=\{ v\in V(G): \deg (v)\geq k\}$ and set $S:=V(G)\setminus L$.  We may assume that $S$ is independent, and that that $a,e\neq 0$. Because of Theorem~\ref{cor:path-1}, we may moreover assume that~$b,d>0$ (and thus~$\ell\geq 2$), and by Lemma~\ref{cor:path-2}, that $c$ is odd. Set $\minae:=\min\{a,e\}$ and set $\minAE:=\max\{a,e\}$.

Suppose that $T\not\subseteq G$. Using the same shifting arguments as
in the proof of Lemma~\ref{cor:path-2}, we may assume that every path in $G$
 of length at least~$k$ zigzags between~$L$ and~$S$, except possibly on its first $a$ and its last~$e$ edges. 
 In fact, as $c$ is odd, we can even assume that every path in $G$ of length at least~$k-1$ zigzags between~$L$ and~$S$, except possibly on its first $a$ and its last $e$ edges.

 As paths are symmetric,  we may actually assume that every path $Q=x_0\ldots x_m$ in $G$ of any length~$m\geq k-1$  zigzags on its subpaths $x_{\minae} Qx_{m-\minAE}$ and $x_{\minAE} Qx_{m-\minae}$. Observe that these subpaths overlap exactly if $\minAE\leq m/2$. Our aim is now to find a path that does not zigzag on the specified subpaths, which will yield a contradiction.

So, let $\mathcal Q$ be the set of those subpaths of $G$ that have length at least $k-1$ and end in $L$. Observe that by Theorem~\ref{cor:path-1}, and since $S$ is independent, $\mathcal Q\neq\emptyset$. Among all paths in $\mathcal Q$, choose $Q=x_0\dots x_{m}$ so that it has a  maximal number of vertices in $L$.

This choice of $Q$ guarantees that $N(x_m)\subseteq S\cup V(Q)$.
Observe that the remark after the statement of Lemma~\ref{cor:path-3} implies that in both cases~(i) and (ii) ,

\begin{align*}
\deg(x_m)\geq k&> \lfloor n/2\rfloor +2\minae\\
 &\geq |S|+2\minae.
\end{align*}

  Since $\minae >0$, we thus obtain that~$x_m$ has a neighbour~$x_s\in L\cap V(Q)$ with $$s\in [ \minae , m-\minae -1].$$

Moreover, in the case that  $\minAE> m/2 $, condition~(ii) of Lemma~\ref{cor:path-3} implies that
\begin{align*}
\deg (x_m)\geq k& \geq 2 (\lfloor n/4\rfloor +a+e+1) -k\\
&\geq\lfloor n/2 \rfloor -1 +2\minae+2\minAE +2 -(m+1)\\
&=|S| +2\minae+2\minAE-m.
\end{align*}
Hence, in this case we can guarantee that
\[
s\in[\minae, m-\minAE -1]\cup [\minAE,m-\minae -1].
\]

Now, consider the path $Q^\ast$ we obtain from $Q$ by joining the subpaths $x_1Qx_s$ and $x_{s+1}Qx_m$ with the edge $x_sx_m$. Then $Q^\ast$ is a path of length $m\geq k-1$ which contains the $L-L$ edge $x_sx_m$. Note that $x_sx_m$ is neither one of the first~$\minae$ nor of the last~$\minae$ edges on $Q^\ast$. Furthermore, in the case that  $\minAE>m/2$, we know that $x_sx_m$ is none of the middle $2\minAE-m$ edges on $Q^*$.  This contradicts our assumption that every path of length at least $k-1$ zigzags between $L$ and $S$ except possibly on these subpaths.
\end{proof}

It remains to prove  Theorem~\ref{prop:2stars}.

\begin{proof}[Proof of Theorem~\ref{prop:2stars}]
Assume we are given graphs $G$ and $T\in\mathcal T(k,\ell,c)$ as in Theorem~\ref{prop:2stars}. If $c$ is even,  it follows from Lemma~\ref{cor:path-2} that $T\subseteq G$. So assume that $\ell + c\geq \lfloor n/2\rfloor $. We shall now use Lemma~\ref{cor:path-3} to see that $T\subseteq G$. Suppose that $T=C(a,b,c,d,e)$.
For $\max\{ a,e\}\leq k/2$, it suffices to observe that $$k-2\min\{a,e\}\geq k-a-e= \ell+c\geq \lfloor n/2\rfloor.$$ On the other hand, for $\max\{ a,e\}> k/2$, observe that $$k-a-e=\ell +c\geq \lfloor n/2\rfloor \geq \lfloor n/4\rfloor+1,$$ since we may assume that $n\geq 4$, as otherwise Theorem~\ref{prop:2stars} holds trivially.
\end{proof}

\bibliographystyle{plain}
\bibliography{graphs}

\begin{thebibliography}{10}

\bibitem{aks}
M.~Ajtai, J.~Koml\'os, and E.~Szemer\'edi.
\newblock On a conjecture of {L}oebl.
\newblock In {\em Proc.\ of the 7th {I}nternational {C}onference on {G}raph
  {T}heory, {C}ombinatorics, and {A}lgorithms}, pages 1135--1146, Wiley, New
  York, 1995.

\bibitem{Barr}
O.~Barr and R.~Johansson.
\newblock {A}nother {N}ote on the {L}oebl--{K}oml\'os--{S}\'os {C}onjecture.
\newblock Research reports no.~22, (1997), Ume\r a University, Sweden.

\bibitem{blw}
C.~Bazgan, H.~Li, and M.~Wo\'zniak.
\newblock On the {L}oebl-{K}oml\'os-{S}\'os conjecture.
\newblock {\em J.~Graph Theory}, 34:269--276, 2000.

\bibitem{diestelBook05}
R.~Diestel.
\newblock {\em Graph Theory \emph{(3rd ed.)}}.
\newblock Springer-Verlag, 2005.

\bibitem{discrepency}
P.~Erd{\H{o}}s, Z.~F\"uredi, M.~Loebl, and V.~T. S\'os.
\newblock Discrepency of trees.
\newblock {\em Studia sci.~Math.~Hungar.}, 30(1-2):47--57, 1995.

\bibitem{gegy}
L.~Gerencs\'er and A.~Gy\'arf\'as.
\newblock On {R}amsey-type problems.
\newblock {\em \em Ann. Univ. Scientiarum Budapestinesis, E\"otv\"os Sect.
  Math.}, 10:167--170, 1967.

\bibitem{harary}
F.~Harary.
\newblock Recent results on generalized {R}amsey theory for graphs.
\newblock In Y.~Alavi et~al., editor, {\em Graph {T}heory and {A}pplications},
  pages 125--138. Springer, Berlin, 1972.

\bibitem{DM}
D.~Piguet and M.~Stein.
\newblock An approximate version of the {L}oebl--{K}oml\'os--{S}\'os
  conjecture.
\newblock Submitted (arXiv:0708.3355v1).

\bibitem{Sun}
L.~Sun.
\newblock On the {L}oebl--{K}oml\'os--{S}\'os conjecture.
\newblock {\em \em Australas. J. Combin.}, 37:271--275, 2007.

\bibitem{zhao}
Y.~Zhao.
\newblock Proof of the $(n/2-n/2-n/2)$ conjecture for large $n$.
\newblock Preprint.

\end{thebibliography}

\end{document}